\documentclass[11pt, letterpaper,final]{amsart}
\usepackage{amsfonts,amsmath, amsthm, amssymb, latexsym, epsfig}
\usepackage[english]{babel}
\usepackage[utf8]{inputenc}
\usepackage[all]{xy}
\usepackage{xspace}
\usepackage{comment}
\usepackage{setspace}
\usepackage{enumerate}
\usepackage{stmaryrd}
\usepackage{xcolor}
\usepackage{hyperref}
\usepackage[mathscr]{eucal}
\usepackage[notcite,notref]{showkeys}
\usepackage{enumitem}

	\newcommand{\ZZ}{\mathbb{Z}}
	\newcommand{\CC}{\mathbb{C}}

	\newcommand{\NN}{\mathbb{N}}

	\newcommand{\Kk}{\mathcal{K}}
    \newcommand{\V}{\mathcal{V}}
    \newcommand{\D}{\mathcal{D}}
    \newcommand{\Z}{\mathcal{Z}}
        \newcommand{\N}{\mathcal{N}}

    \newcommand{\spn}{\operatorname{span}}

\newtheorem{theoremintro}{Theorem}

\newtheorem*{questionintro}{Question}

	\newtheorem{thm}{Theorem}[section]
	\newtheorem{cor}[thm]{Corollary}
	\newtheorem{lem}[thm]{Lemma}
	\newtheorem{prp}[thm]{Proposition}

	\theoremstyle{definition}
	\newtheorem{dfn}[thm]{Definition}

	\theoremstyle{remark}
	\newtheorem{rmk}[thm]{Remark}

\begin{document}
	
    \title[Diagonal-preserving isomorphisms]{Diagonal-preserving isomorphisms of algebras from infinite graphs}
    

 \author{S{\o}ren Eilers}
        \address{Department of Mathematical Sciences \\
        University of Copenhagen\\
        Universitetsparken~5 \\
        DK-2100 Copenhagen, Denmark}
        \email{eilers@math.ku.dk }

	\author{Efren Ruiz}
        \address{Department of Mathematics\\University of Hawaii,
Hilo\\200 W. Kawili St.\\
Hilo, Hawaii\\
96720-4091 USA}
        \email{ruize@hawaii.edu}

        \date{\today}
\thanks{This research is part of the EU Staﬀ Exchange project 101086394 \emph{Operator Algebras That One
Can See}. E. Ruiz was partially supported by a Collaborative-NSF-BSF grant (DMS 2452325)}

\begin{abstract}
We establish logical equivalence between statements involving
\begin{itemize}
\item the Cuntz $C^*$-algebra $\mathcal O_\infty$ with its canonical diagonal;
\item graph $C^*$-algebras  with their canonical diagonals;
\item Leavitt path algebras over general fields with their canonical diagonals;
\item Leavitt path algebras over $\mathbb Z$;
\item  topological full groups;
\item  groupoids; and
\item the automorphism  $x\mapsto -x$ on certain $K_0$- and homology groups equal to $\mathbb Z$.
\end{itemize}
Deciding whether these equivalent statements are true or false is of importance in   studies of geometric classification of diagonal-preserving isomorphism between graph $C^*$-algebras and Leavitt path algebras, mirroring a similar hindrance studied by Cuntz more than 40 years ago.
\end{abstract}

\dedicatory{Dedicated to Professor Gene Abrams on the occasions of his retirement and 70th birthday.}

\maketitle

\section{Introduction}
After having taken the first step, with Krieger \cite{CK80}, in what would evolve to become the complete classification of unital graph $C^*$-algebras, Joachim Cuntz drew attention to the pair of graphs

\medskip
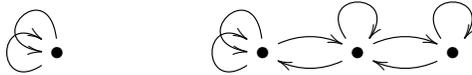
\[
\xymatrix{\bullet\ar@(u,l)[]\ar@(dl,ul)[]} \qquad\qquad\qquad \xymatrix{\bullet\ar@(u,l)[]\ar@(dl,ul)[]\ar@/^/[r]&\bullet\ar@(ul,ur)[]\ar@/^/[l]\ar@/^/[r]&\bullet\ar@(ul,ur)[]\ar@/^/[l]}
\]
\medskip

\noindent at the 1983 Kyoto conference ``Geometric methods in Operator Algebras'' \cite{C86}.  The $C^*$-algebra defined by the graph to the left is just the Cuntz algebra $\mathcal O_2$, and introducing the notation $\mathcal O_{2,-}$ for the $C^*$-algebra defined by the graph to the right, Cuntz pointed out that it was unknown if they would define the same $C^*$-algebra. He further  explained how an answer to this question was a natural stepping stone towards a $K$-theoretic
classification of simple Cuntz-Krieger algebras, and also showed that the isomorphism would follow if one could realize the unique non-trivial unit-preserving  isomorphism of the $K$-theory of the $C^*$-algebra given by the graph

\medskip
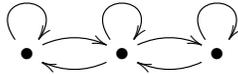
\[
\xymatrix{\bullet\ar@(ul,ur)[]\ar@/^/[r]&\bullet\ar@(ul,ur)[] \ar@/^/[l]\ar@/^/[r] & \bullet  \ar@/^/[l]\ar@(ul,ur)[]}
\]
\medskip
\noindent by a $*$-automorphism.

Cuntz subsequently developed the ideas presented in \cite{C86} to the effect that when R\o{}rdam showed $\mathcal O_2$ and $\mathcal O_{2,-}$ are indeed isomorphic in the mid-nineties \cite{Ror95}, the classification of simple Cuntz-Krieger algebras was instantly completed.

And although these events predate by decades the definition of Leavitt path algebras that we celebrate here along with Gene Abrams, Cuntz' question remains famously open in the form that we still do not know whether there is a field $\mathsf{k}$ for which the Leavitt path algebras $L_2^{\mathsf{k}}$ and $L_{2,-}^{\mathsf{k}}$ are isomorphic.  In analogy to Cuntz' work, Abrams \emph{et al.} \cite{AALP08, ALPS11} and Ruiz \cite{RuizBLMS25}, showed the role of this isomorphism question for Leavitt path algebras to the classification of simple Leavitt path algebras.

In the endeavor to understand structure-preserving isomorphism amongst unital graph $C^*$-algebras and Leavitt path algebras, we are finding ourselves facing a similar impasse concerning the pair of graphs with infinitely many edges

\medskip

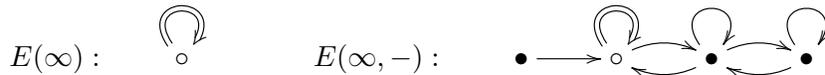
\[
E(\infty): \qquad \xymatrix{\circ\ar@(ul,ur)@{=>}[]} \qquad\qquad E(\infty,-) : \qquad  \xymatrix{\bullet\ar[r]&\circ\ar@(ul,ur)@{=>}[]\ar@/^/[r]&\bullet\ar@(ul,ur)[]\ar@/^/[l]\ar@/^/[r]&\bullet\ar@(ul,ur)[]\ar@/^/[l]}
\]
\medskip

\noindent which define the standard $C^*$-algebras and Leavitt path algebras $\mathcal O_\infty$ and $L_\infty^{\mathsf{k}}$ on one hand, and what we will call $\mathcal O_{\infty,-}$ and $L_{\infty,-}^{\mathsf{k}}$ on the other.  We are concerned with the following questions.    

\begin{questionintro}
$ $
\begin{enumerate}
\item Is there a diagonal-preserving isomorphism between $\mathcal O_\infty$ and $\mathcal O_{\infty,-}$?

\item Is there a diagonal-preserving isomorphism between $L_{\infty}^{\mathsf{k}}$ and $L_{\infty,-}^{\mathsf{k}}$?
\end{enumerate}
\end{questionintro}

Pioneered by Kengo Matsumoto \cite{KMbook}, considering graph $C^*$-algebras and Leavitt path algebras together with their respective diagonal subalgebras have shown to have profound rigidity properties.  It has led to the celebrated result that flow equivalence of shift spaces of finite type is characterized by the isomorphism class of the stabilized graph algebra together with its diagonal subalgebra \cite{MM13, CEOR} and the result that $L_2^\mathbb{Z}$ and $L_{2,-}^\mathbb{Z}$ are not isomorphic as $*$-algebras \cite{JS16,TCadv}.  Because of the latter and the fact that all known classification results of Leavitt path algebras \cite{ AALP08,AAP08,ALPS11} can be shown to be a classification of Leavitt path algebras with their diagonal subalgebras, one can speculate that the reason a proof that $L_2^\mathsf{k}$ and $L_{2,-}^\mathsf{k}$ are isomorphic has eluded us for over a decade is that these Leavitt path algebras are not isomorphic.  Furthermore, considering graph algebras with their diagonal subalgebras have allowed researchers to obtain a Rosetta Stone between graph $C^*$-algebras and Leavitt path algebras, where properties shared by both can be simultaneously proved using their associated groupoids.

As the main new contribution in this note, we show a result parallelling \cite{C86} in reducing the isomorphism question to a question concerning the $*$-automorphic realization of the unique non-trivial automorphism on the $K$-theory of $\mathcal O_\infty\otimes\mathbb K$. But in this instance, however, many things differ from Cuntz' setup. First, we know from the outset that $\mathcal O_\infty$ and $\mathcal O_{\infty,-}$ are isomorphic as $C^*$-algebras (\cite{ERRSMathAnnalen, ER-Refined}); the open question is whether they are isomorphic in a way preserving the diagonal subalgebra. This in turn, by existing work, can be translated into questions on groupoids or topological full groups, and we also know that our base question is equivalent to the same question for Leavitt path algebras, and to an isomorphism question for the Leavitt path algebras over $\mathbb Z$.  The difficulty here is that there is a diagonal-preserving isomorphism between $\mathcal O_\infty \otimes \mathcal{K}$ and $\mathcal O_{\infty,-} \otimes \mathcal{K}$ and there is a diagonal-preserving isomorphism between $\mathsf{M}_\mathbb{N}(L_\infty^{\mathsf{k}})$ and $\mathsf{M}_\mathbb{N}(L_{\infty,-}^{\mathsf{k}})$ \cite{Sor13, RT13,ER-Refined}, but not obviously one before stabilization.

Answering the above questions have consequences for a conjecture made by the authors in \cite{ER-Refined}.  We conjectured that two graph $C^*$-algebras are isomorphic via an isomorphism that preserves the diagonal subalgebras if and only if the Leavitt path algebras are isomorphic via an isomorphism that preserves the diagonal subalgebras if and only if the associated graphs are equal in the coarsest equivalence relation containing the unital flow moves.  We showed in \cite{ER-Refined} that a positive answer to the above questions implies our generating conjecture would be false as we have shown that one can not move $E(\infty)$ to $E(\infty, -)$ using the unital flow moves.

For a graph $E$, $\Sigma E$ will denote the graph where for each $v \in E^0$, we attach a countably infinite fan at $v$,
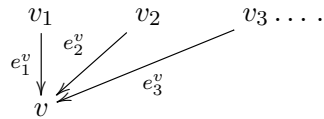
\[
\xymatrix{ v_1 \ar[d]_-{e_1^v} &  \ar[dl]_-{e_2^v} v_2  & \ar[dll]^-{e_3^v} v_3\ldots    \\ 
         v}.
\]
For an \'etale groupoid $\mathcal{G}$, its \emph{topological full group}, $[[\mathcal{G}]]$, is the set of all homeomorphism $\alpha$ on $\mathcal{G}^{(0)}$ such that there exists an open set $U$ of $\mathcal{G}$ for which the range and source maps of $G$ are injective on $U$, $r(U)=s(U)=\mathcal{G}^{(0)}$, $\alpha (s(g)) = r(g)$ for all $g \in U$, and $\operatorname{supp} \alpha = \overline{\{ x \in \mathcal{G}^{(0)} : \alpha(x) \neq x \} }$ is compact.

\begin{theoremintro}\label{thm-intro}
The following are equivalent.
\begin{enumerate}[label=\textbf{D.\arabic*}]
    \item \label{EF-cstar} The graph $C^*$-algebras $\mathcal{O}_\infty$ and $\mathcal{O}_{\infty,-}$ are isomorphic via an isomorphism that preserves the diagonal subalgebras.

    \item \label{cartan} If $\mathcal{O}_\infty$ is isomorphic to $C^*(H)$, then there is a diagonal-preserving isomorphism between $\mathcal{O}_\infty$ and $C^*(H)$.

    \item \label{EF-LPA} The Leavitt path algebras $L_\infty^{\mathsf{k}}$ and $L_{\infty,-}^\mathsf{k}$ are isomorphic via an isomorphism that preserves the diagonal subalgebras.

    \item \label{EF-LPA-Z} The Leavitt path algebras $L_\infty^{\mathbb{Z}}$ and $L_{\infty, -}^\mathbb{Z}$ are $*$-isomorphic.

    \item \label{EF-Groupoid} The graph groupoids $\mathcal{G}_{E(\infty)}$ and $\mathcal{G}_{E(\infty,-)}$ are isomorphic.

    \item \label{EF-fullgrps} The topological full groups $[[\mathcal{G}_{E(\infty)}]]$ and $[[\mathcal{G}_{E(\infty,-)}]]$ are isomorphic.

    \item \label{Oinfty} There exists a diagonal-preserving stable automorphism of $\mathcal{O}_\infty$ that induces the non-trivial automorphism on $K_0(\mathcal{O}_\infty)$.

    \item \label{Linfty} There exists a diagonal-preserving stable automorphism of $L_\infty^{\mathsf{k}}$ that induces the non-trivial automorphism on $K_0(L^\mathsf{k}_\infty)$.

    \item \label{Linfty-Z} There exists a stable $*$-automorphism of $L_{\infty}^\mathbb{Z}$ that induces the non-trivial automorphism on $K_0(L^\mathbb{Z}_\infty)$.

    \item \label{Groupoid} There exists a groupoid automorphism on $\mathcal{G}_{\Sigma E(\infty)}$ that induces the non-trivial automorphism on the zeroth groupoid homology.
\end{enumerate}
\end{theoremintro}

The fact that \eqref{EF-cstar}, \eqref{EF-LPA}, \eqref{EF-LPA-Z}, \eqref{EF-Groupoid}, and \eqref{EF-fullgrps} are equivalent statements follows from \cite[Corollary~4.3 and Corollary~4.4]{ABHS}, \cite[Theorem~1]{TCadv}, and \cite[Theorem~D]{NO2019}.  Our main tool to complete the proof is to give a new description of the abelian monoid of Murray-von Neumann equivalence classes of idempotents in $\mathsf{M}_\mathbb{N}(L_\mathsf{k}(E))$ (resp. in $C^*(E) \otimes \Kk$).  We will show this monoid is generated by idempotents in the diagonal subalgebra of $\mathsf{M}_\mathbb{N}(L_\mathsf{k}(E))$ (resp. $C^*(E) \otimes \Kk$) subject to the relation of Murray-von Neumann equivalence via partial isometries that normalize the diagonal subalgebra.  This generalizes the result of Matsumoto \cite{KM2013} for which he proved the same result for strongly connected graphs.

\section{Extension of Matsumoto's result}

Let $A$ be an irreducible nonpermutation matrix with nonnegative integer entries, let $\mathcal{O}_A$ be the Cuntz-Krieger algebra \cite{CK80}, and let $\D_A$ be its canonical diagonal subalgebra.  In \cite{KM2013}, Matsumoto proves that if $p$ and $q$ are projections in the subalgebra $\mathcal{D}_A \otimes c_0(\NN)$ of $\mathcal{O}_A \otimes \Kk$ consisting of diagonal matrices with entries in $\D_A$, then $[p] = [q]$ in $K_0( \mathcal{O}_A \otimes \Kk )$ if and only if there exists a partial isometry $v \in \mathcal{O}_A \otimes \Kk$ such that $v^*v = p$, $vv^* = q$, and 
$$
v (\mathcal{D}_A \otimes c_0(\NN)) v^* \subseteq \mathcal{D}_A \otimes c_0(\NN) \quad \text{and} \quad v^* (\mathcal{D}_A \otimes c_0(\NN))v \subseteq \mathcal{D}_A \otimes c_0(\NN).
$$
In this section, we will generalize this result to Leavitt path algebras and graph $C^*$-algebras.

\begin{dfn}
Let $A$ be a $*$-ring and let $D$ be a $*$-subring of $A$.  An element $v \in A$ is a \emph{normalizer of $D$} provided that $v D v^* \subseteq D$ and $v^* D v \subseteq D$. The set of \emph{normalizers of $D$} will be denoted by 
\[\N^*( A, D) := \{ v \in A \ \mid \  v D v^* \subseteq D, v^* D v \subseteq D \}.\]
For projections $p, q \in D$, we write $p \underset{D}{\sim} q$ provided that there exists $v \in \mathcal{N}^*(A, D)$ such that $v^*v = p$ and $vv^*=q$.  
\end{dfn}

A computation shows that $\underset{D}{\sim}$ is an equivalence relation on the set of projections in $D$.  For each projection $p$ in $D$, denote the equivalence class of $p$ by $[p]_{D}$.  

For a ring $A$, $\mathsf{M}_\NN(A)$ denotes the ring of all infinite matrices indexed by $\NN$ with all but finitely many of entries are zero.  For a unital ring $A$, let $\{ e_{i,j} \}$ be the standard system of matrix units in $\mathsf{M}_\NN(A)$, that is, $e_{i,j}$ is the element in $\mathsf{M}_\NN(A)$ whose $(i,j)$ entry is $1$ and all other entries are zero.  For an element $a$ of $A$, $a \otimes e_{i,j}$ will denote the element of $\mathsf{M}_{\NN} (A)$ with $a$ in its $(i,j)$ entry and zero in all other entry.  Let $D^s$ be the $*$-subring of $\mathsf{M}_{\NN} (A)$ consisting of all diagonal matrices with entries in $D$.  Note that every projection in $D^s$ is equal to $\sum_{k = 1}^m p_k \otimes e_{k,k}$, where $p_k$ is a projection in $D$.  For $p = \sum_{ k = 1}^m p_k \otimes e_{k,k}$ and $q = \sum_{k=1}^n q_k \otimes e_{k,k}$ projections in $D^s$, we denote the \emph{orthogonal sum} by 
$$p \oplus q:= \sum_{ k = 1}^m p_k \otimes e_{k,k} + \sum_{ k = 1}^n q_k \otimes e_{m+k, m+k }.$$
Since $p_k$ and $q_k$ are projections in $D$, $p\oplus q$ is a projection in $D^s$.  
We set $\N_s^*( A, D) := \N^* ( \mathsf{M}_\NN(A), D^s)$.

\begin{dfn}
Let $A$ be a $*$-ring and let $D$ be a $*$-subring of $A$.  Let $\V^*( A , D)$ to be the set 
$$\left\{ [ p ]_{ D^s } \ \mid \ p \text{ a projection in } D^s\right\}.$$
\end{dfn}
A computation shows that $\V^*( A , D)$ with orthogonal sum is an abelian monoid. 

\begin{prp}\label{lem:rel-ktheory}
Let $A$ be a $*$-ring and let $D$ be a $*$-subring of $A$.  Defining addition on $\V^*(A,D)$ by
$$[p]_{D^s } + [q]_{D^s} := [ p \oplus q ]_{ D^s },$$
$\V^*(A,D)$ becomes an abelian monoid.
\end{prp}

We now prove two lemmas about $\V^*(A,D)$ that we will use in the proof of Theorem~\ref{thm:monoid}.  The first lemma holds for a general $*$-ring and the second specializes to $L_\ZZ(E)$.

\begin{lem}\label{lem:11todiag}
Let $A$ be a $*$-ring and let $D$ be a commutative $*$-subring of $A$.  Suppose $p_1, p_2, \ldots, p_n$ are projections in $D$ such that $p_i p_j = 0$ for all $i \neq j$.  Then $\left[ \left( \sum_{i=1}^n p_i \right) \otimes e_{1,1}\right]_{D^s}= \left[ \sum_{ i = 1}^n p_i \otimes e_{i, i}\right]_{D^s}$.   
\end{lem}

\begin{proof}
Set $S = \sum_{ i = 1}^n  p_i \otimes e_{i, 1}$.  Then $S^* =  \sum_{ i = 1}^n  p_i \otimes e_{1, i}$.  Note that $S, S^*$ are elements of $\mathsf{M}_{\NN}(A)$ and 
\begin{align*}
S^*S &= \sum_{ i, j = 1}^n p_i p_j \otimes e_{1,i} e_{j,1} & SS^* &= \sum_{ i, j = 1 }^n p_ip_j \otimes e_{i,j} \\
    &= \sum_{ i = 1}^n p_i \otimes e_{1,1}  & &= \sum_{ i = 1}^n p_i \otimes e_{i,i}. \\
    &= \left( \sum_{i=1}^n p_i \right) \otimes e_{1,1}
\end{align*}
Let $x = \sum_{ i = 1}^\infty a_i \otimes e_{i,i} \in D^s$, where all but finitely many of the $a_i$s are zero.  Then
\begin{align*}
S x S^* &= \sum_{i, j = 1}^n (p_i \otimes e_{i,1} ) x ( p_j \otimes e_{1,j} ) =  \sum_{ i , j = 1}^n \sum_{ k = 1}^\infty (p_i \otimes e_{i,1 } )(a_k \otimes e_{k,k}) (p_j \otimes e_{1,j})\\
        &= \sum_{i, j = 1}^n p_i a_1 p_j \otimes e_{i,j} = \sum_{ i, j=1}^n p_i p_j a_1 \otimes e_{i,j} = \sum_{ i =1}^n p_i a_1 \otimes e_{i, i} \in D^s \quad \text{and} \\
S^* x S &= \sum_{ i , j = 1}^n (p_i \otimes e_{1,i } )x(p_j \otimes e_{j,1}) = \sum_{ i , j = 1}^n \sum_{ k = 1}^\infty (p_i \otimes e_{1,i } )(a_k \otimes e_{k,k}) (p_j \otimes e_{j,1}) \\
        &= \sum_{ i = 1}^n p_i a_i p_i \otimes e_{1,1} \in D^s.
\end{align*}
Therefore, $S \in \mathcal{N}_s^*(A,D)$ which completes the proof that \[\left[ \left( \sum_{i=1}^n p_i \right) \otimes e_{1,1}\right]_{D^s}= \left[ \sum_{ i = 1}^n p_i \otimes e_{i, i}\right]_{D^s}.\]
\end{proof}

\begin{dfn}
Let $E$ be a graph.  Then the commutative subring  
$$\spn \{ \mu \mu^*  \mid  \mu \text{ path in } E \}$$ 
of $L_R(E)$ is called the \emph{diagonal} subring of $L_R(E)$ and is denoted by $\D_R(E)$.  The closure of $\spn_\CC \{ \mu \mu^*  \mid  \mu \text{ path in } E \}$ in $C^*(E)$ will be denoted by $\D(E)$.
\end{dfn}
To simplify notation, set $\D_R^s(E) := (\D_R(E))^s$ of diagonal matrices in $\mathsf{M}_\NN(L_R(E))$ with entries in $\D_R(E)$.  When \(R=\ZZ\), $\D_\ZZ(E)$ is a $*$-subring of $L_\ZZ(E)$.  For ease of notation, elements of $\V^* ( L_\ZZ(E), \D_\ZZ(E))$ will be denoted by $[p]_{\D}$ instead of $[p]_{\D_\ZZ^s(E)}$, where $p$ is a projection in $\D_\ZZ^s(E)$.

\begin{lem}\label{lem:range-source-proj}
Let $E$ be a graph and let $\mu_1, \mu_2, \ldots, \mu_n$ be paths in $E$.  Then 
$$\left[ \sum_{ k = 1}^n \mu_k \mu_k^* \otimes e_{k,k} \right]_{ \D } = \left[ \sum_{ k = 1}^n r(\mu_k) \otimes e_{k,k} \right]_{ \D}$$
in $\V(L_\ZZ(E), \D_\ZZ(E))$.
\end{lem}

\begin{proof}
Let $\mu$ be a path in $E$.  By a similar argument as in \cite[Lemma~4.1]{BCW2017}, $\mu \in \N^*( L_\ZZ(E), \D_\ZZ(E))$.  Set $V = \sum_{ k = 1}^n \mu_k \otimes e_{k,k}$ and let $x = \sum_{k = 1}^\infty x_k \otimes e_{k,k} \in \D$, where $x_k \in \D_\ZZ(E)$ and all but finitely many of the $x_k$s are zero.  Then
\[V x V^* = \sum_{k=1}^n \mu_k x_k \mu_k^* \otimes e_{k,k}   \quad \text{and} \quad V^* x V = \sum_{k=1}^n \mu_k^* x_k \mu_k \otimes e_{k,k}\]
are elements of $\D^s_\ZZ(E)$.  Therefore, $V \in \N_s^*( L_\ZZ(E), \D)$.  Since
\begin{align*}
V^*V &= \sum_{ k = 1}^n \mu_k^* \mu_k \otimes e_{k,k} = \sum_{ k = 1}^n r(\mu_k) \otimes e_{k,k} \quad \text{and} \\
VV^* &= \sum_{ k = 1}^n \mu_k \mu_k^* \otimes e_{k,k},
\end{align*}
we have $\left[ \sum_{ k = 1}^n \mu_k \mu_k^* \otimes e_{k,k} \right]_{ \D} = \left[ \sum_{ k = 1}^n r(\mu_k)\otimes e_{k,k} \right]_{ \D}$.
\end{proof}

To prove Theorem~\ref{thm:monoid}, we will need to describe the idempotents of $\D^s_R(E)$ for an integral domain $R$ (Lemma~\ref{lem:diag-proj-graph}).  In order to do this, we need the graph groupoid.  For a graph $E$, let $\partial E$ be the union of all finite paths ending at a singular vertex and all infinite paths.  The sets
\[\mathcal{Z}(\mu \setminus F) := \{ \mu x \in \partial E : x \in \partial E\} \setminus \left( \bigcup_{ e \in F } \{ \mu e y \in \partial E : y \in \partial E \} \right)\]
generate a topology on $\partial E$ that makes $\partial E$ a locally compact, Hausdorff space.  The \emph{graph groupoid}, $\mathcal{G}_E$, is defined as a set
\[
\mathcal{G}_E = \{ (x, m-n, y) \in \partial E \times \ZZ \times \partial E : |x| \geq m, |y| \geq n, \sigma_E^m(x) = \sigma^n(y) \},
\]
where $\sigma_E \in \partial E^{\geq 1} \to \partial E$ is the \emph{shift map} that removes the first edge in the path.  The product is defined by $(x,k, y)(y,l,z) = (x, k+l, z)$.  We give $\mathcal{G}_E$ the topology generated by 
\[
\mathcal{Z}_E(U, m, n, V) = \{ (x,m-n,y) \in \mathcal{G}_E : x \in U, y \in V, \sigma_E^m(x)=\sigma_E^n(y) \},
\]
where $U$ and $V$ are open subsets of $\partial E$ such that $\sigma_E^m \vert_{U}$ and $\sigma_E^n \vert_{V}$ are injections, and $\sigma_E^m(U)=\sigma_E^n(V)$.

\begin{lem}\label{lem:diag-proj-graph}
Let $E$ be a graph and let $R$ be an integral domain.  Let $p$ be a nonzero idempotent in $\D_R(E)$.  Then there are paths $\mu_1, \mu_2, \ldots, \mu_n$ in $E$ and finite sets $S_1, S_2, \ldots, S_n$ of $E^1$ such that 
\begin{enumerate}
    \item each $S_i$ is a subset of $r(\mu_i) E^1$;

    \item $\left\{\mu_i \mu_i ^* - \sum_{ e \in S_i } (\mu_i e) (\mu_i e)^* : i =1, 2, \ldots, n \right\}$ is a collection of mutually orthogonal idempotents; and 

    \item $p = \sum_{ i = 1}^n \left(\mu_i \mu_i^* - \sum_{ e \in S_i } (\mu_i e) (\mu_i e)^* \right)$.
\end{enumerate}
In particular, every idempotent in $\D_R(E)$ is the image of a projection in $L_\ZZ(E)$.
\end{lem}

\begin{proof}
By \cite[Example~3.2]{ClarkSims}, there exists an isomorphism from $L_R(E)$ to the Steinberg algebra $A_R(\mathcal{G}_E)$ sending $\D_R(E)$ to $A_R( \mathcal{G}_E^0)$ such that $\mu \mu^*$ is sent to $1_{\Z(\mu)}$.  Using this isomorphism, we need to prove that for all nonzero idempotents $p$ in $A_R( \mathcal{G}_E^0)$, there are paths $\mu_1, \mu_2, \ldots, \mu_n$ in $E$ and finite sets $S_1, S_2, \ldots, S_n$ of $E^1$ such that 
\begin{enumerate}
    \item each $S_i$ is a subset of $r(\mu_i) E^1$;

    \item $\{ \Z (\mu_i) \setminus S_i \}_i$ are mutually orthogonal sets; and 

    \item  $p =\sum_{ i =1}^n 1_{\Z(\mu_i)\setminus F_i}$.
\end{enumerate}
Let $p$ be a nonzero idempotent in $A_R( \mathcal{G}_E^0)$.  Then $p = 1_K$ for some compact, open subset of $\mathcal{G}_E^0$.  By \cite[Proof of Lemma~2.1]{BCW2017}, 
$$K= \bigsqcup_{ x \in T } \Z( \mu_x \setminus S_x ),$$
where $S_x$ is a finite subset $r(\mu_x) E^1$.  Since $K$ is a compact subset of $\mathcal{G}_E^0$, we may assume that $T$ is a finite set.  Then $\{ \mu_x \}_{x \in T}$, $\{ S_x \}_{ x \in T }$, and $\{ \Z( \mu_x \setminus S_x) \}_{x \in T}$ are the desired collections of paths and sets.
\end{proof}

In \cite{AraGoodearl2012}, Ara and Goodearl constructed a monoid $M(E, C, S)$ for any separated graph and they proved that there is a natural isomorphism between $M(E, C, S)$ and the $\V$-monoid of the separated Leavitt path algebra.  We now describe their monoid in the unseparated case, which we denote by $M_E$.  The monoid $M_E$ is generated by $v \in E^0$ and $q_S$, $S$ runs through all nonempty finite subsets of $v E^1$ for infinite emitters $v$ and is subject to the relations
\begin{itemize}
    \item[(i)] $v = \sum_{ e \in v E^1 } r(e)$ for all regular vertices $v$; 

    \item[(ii)] $v = \sum_{ e \in S } r(e) + q_S$ for all infinite emitters $v$; and 

    \item[(iii)] $q_{S_1} = q_{S_2}+ \sum_{e \in S_2 \setminus S_1 } r(e)$ for all nonempty subsets $S_1 \subseteq S_2 \subseteq vE^1$, where $v$ is an infinite emitter.  
\end{itemize}
For any field $\mathsf{k}$, the isomorphism $\lambda_E$ from $M_E$ to $\V(L_\mathsf{k}(E))$ is given by $v \mapsto [v]$, $q_S \mapsto w - \sum_{ e \in S } ee^*$, for any $v \in E^0$ and any nonempty finite subset $S$ of $wE^1$ with $w$ an infinite emitter.

For a graph $E$ and for a field $\mathsf{k}$, there exists an injective homomorphism $\iota_E \colon L_\ZZ(E) \to L_\mathsf{k}(E)$ such that $\iota_E$ is the identity map on $E^0$ and $E^1$.

\begin{thm}\label{thm:monoid}
Let $E$ be a graph.  Then the mapping
$$v \mapsto [v \otimes e_{1,1}]_{\D} \quad \text{and} \quad q_S \mapsto \left[ \left(w - \sum_{ e\in S} e e^*\right) \otimes e_{1,1}\right]_{\D}$$ 
where $S$ is a nonempty finite subset of $wE^1$ with $w$ an infinite emitter induces an isomorphism from $M_E$ to $\V^*(L_\ZZ(E), \D_\ZZ(E))$.  Moreover, for any field $\mathsf{k}$, the map $\beta_E$ from $\V^*(L_\ZZ(E), \D_\ZZ(E))$ to $\V(L_\mathsf{k}(E))$ which sends $[p]_\D$ to $[\iota_E(p)]$ is an isomorphism. 
\end{thm}

\begin{proof}
For each infinite emitter $w$ and for each nonempty finite subset $S$ of $wE^1$, set $p_{S} = w - \sum_{ e\in S} e e^*$.  We will show that $a_v:=  [ v \otimes e_{1,1} ]_{ \D } $ and $a_S:= [p_S\otimes e_{1,1} ]_{\D }$ satisfies the relations defining $M_E$.  Let $v$ be a regular vertex.  Set $vE^1 = \{ e_1, \ldots, e_m \}$. 
 Since $e_k e_k^*$'s are mutually orthogonal projections, 
\begin{align*}
a_v &\underset{\phantom{\text{Lem.}~\ref{lem:range-source-proj}}}{=} \left[ \left( \sum_{k=1}^m e_k e_k^* \right) \otimes e_{1,1} \right]_{ \D  } \underset{\text{Lem.}~\ref{lem:11todiag}}{=} \left[ \sum_{ k=1}^m e_k e_k^* \otimes e_{k, k}  \right]_{ \D  } \\
&\underset{\text{Lem.}~\ref{lem:range-source-proj}}{=} \left[ \sum_{ k=1}^m r(e_k) \otimes e_{k,k} \right]_{ \D  } = \sum_{ k = 1}^m [ r(e_k) \otimes e_{1,1}]_{ \D  } = \sum_{ e \in vE^1} a_{r(e)}.
\end{align*}

Let $w$ be an infinite emitter and let $S = \{ f_1, f_2, \ldots, f_n \}$ be a nonempty finite subset of $wE^1$.  As $p_S$ and $f_k f_k^*$'s are mutually orthogonal projections,   
\begin{align*}
a_w &\underset{\phantom{\text{Lem.}~\ref{lem:diag-proj-graph}}}{=} \left[ \left( p_S + \sum_{k=1}^m f_k f_k^* \right) \otimes e_{1,1} \right]_{ \D } \underset{\text{Lem.}~\ref{lem:11todiag}}{=} a_S + \left[  \sum_{k=1 }^m f_k f_k^* \otimes e_{k+1,k+1} \right]_{ \D } \\
&\underset{\text{Lem.}~\ref{lem:range-source-proj}}{=} a_S + \left[  \sum_{k=1 }^n r(f_k) \otimes e_{k+1,k+1} \right]_{ \D } \underset{\phantom{\text{Lem.}~\ref{lem:diag-proj-graph}}}{=} a_S + \sum_{ k = 1}^n a_{ r(f_k)}.
\end{align*}
Let $S_1$ and $S_2$ be nonempty finite subsets of $wE^1$ such that $S_1 \subseteq S_2$.  Set  $S_2 \setminus S_1 = \{g_1, \ldots, g_t \}$.  As $w - \sum_{ e \in S_2 } e e^*$ and $g_kg_k^*$'s are mutually orthogonal projections,
\begin{align*}
a_{S_1}&\underset{\phantom{\text{Lem.}~\ref{lem:diag-proj-graph}}}{=} \left[ \left(w - \sum_{ e \in S_1 } e e^* \right)\otimes e_{1,1} \right]_{\D} \\
&\underset{\phantom{\text{Lem.}~\ref{lem:diag-proj-graph}}}{=} \left[ \left(w - \sum_{ e \in S_2 } e e^* + \sum_{ e \in S_2 \setminus S_1 }  e e^* \right)\otimes e_{1,1} \right]_{\D} \\
&\underset{\text{Lem.}~\ref{lem:11todiag}}{=}a_{S_2}  + \left[ \sum_{ k = 1 }^t  g_k g_k^* \otimes e_{k+1,k+1} \right]_{\D} \\
&\underset{\text{Lem.}~\ref{lem:range-source-proj}}{=}a_{S_2} + \left[ \sum_{ k = 1 }^t  r(g_k) \otimes e_{k+1,k+1} \right]_{\D} \\
&\underset{\phantom{\text{Lem.}~\ref{lem:diag-proj-graph}}}{=} a_{S_2}+ \sum_{ k = 1}^t  a_{r(g_k)}   = a_{S_2} + \sum_{e \in S_2\setminus S_1} a_{r(e)}.    
\end{align*}
Hence, there exists a homomorphism $\alpha_E \colon M_E \to \V^*( L_\ZZ(E), D_\ZZ(E))$ such that $\alpha_E(v)= a_v$ and $\alpha_E( q_S) =a_S$.  It is clear that $[p]_{\D } \mapsto [p]$ defines a homomorphism $\beta_E$ from $\V^*( L_\ZZ(E), D_\ZZ(E))$ to $\V(L_\mathsf{k}(E))$ such that $\beta_E\circ \alpha_E = \lambda_E$.  Since $\lambda_E$ is an isomorphism, $\alpha_E$ is an injection.  We are left to show that $\alpha_E$ is a surjection.

We claim that for all path $\mu$ in $E$ and for any finite subset $S$ of $r(\mu)E^1$, $\left[ \left(\mu \mu^* - \sum_{ e \in S } (\mu e)(\mu e)^*\right) \otimes e_{1,1} \right]_{\D}$ is in the range of $\alpha_E$.  Let $\mu$ be a path in $E$ and let $S$ be a finite subset of $r(\mu)E^1$.  Suppose $r(\mu)$ is a regular vertex.  Set $vE^1 \setminus S = \{ e_1, e_2, \ldots, e_n \}$.  Then 
\begin{align*}
&\left[ \left(\mu \mu^* - \sum_{ e \in S } (\mu e) (\mu e)^*\right) \otimes e_{1,1} \right]_{\D} \\
&\underset{\phantom{\text{Lem.}~\ref{lem:11todiag}}}{=} \left[ \sum_{ e \in r(\mu) E^1 \setminus S }  (\mu e) (\mu e)^* \otimes e_{1,1}\right]_{ \D } \underset{\text{Lem.}~\ref{lem:11todiag}}{=}\left[ \sum_{ k=1 }^n (\mu e_k)(\mu e_k)^* \otimes e_{k,k}\right]_{ \D } \\
&\underset{\text{Lem.}~\ref{lem:range-source-proj}}{=}\left[ \sum_{ k=1 }^n  r(\mu e_k) \otimes e_{k,k}\right]_{ \D }  = \sum_{ k = 1}^n [ r( \mu e_k)  \otimes e_{1,1} ]_{ \D }  =  \sum_{ k = 1}^n \alpha_E ( r(\mu e_k) ).
\end{align*}
Suppose $r(\mu)$ is an infinite emitter and $S = \emptyset$.  Then 
\begin{align*}
\left[ \left(\mu \mu^* - \sum_{ e \in S } (\mu e)(\mu e)^*\right) \otimes e_{1,1} \right]_{\D} &= [ \mu \mu^* \otimes e_{1,1}]_{\D}  \underset{\text{Lem.}~\ref{lem:range-source-proj}}{=} [ r(\mu) \otimes e_{1,1}]_{\D} \\
&=  \alpha_E ( r(\mu) ).
\end{align*}
Suppose $r(\mu)$ is an infinite emitter and $S \neq \emptyset$.  Set 
\[V = \mu^*  \left(\mu \mu^* - \sum_{ e \in S } (\mu e)(\mu e)^*\right) \otimes e_{1,1}.\]
Then $V^*=  \left(\mu \mu^* - \sum_{ e \in S } (\mu e)(\mu e)^*\right)\mu \otimes e_{1,1}$.  Using the same argument as \cite[Proof of Lemma~4.1]{BCW2017}, $\mu$ is an element of $\N^*(L_\ZZ(E), \D_\ZZ(E))$.  And since $\mu \mu^* - \sum_{ e \in S } (\mu e)(\mu e)^* \in \D_\ZZ(E)$, a computation shows that $V$ is an element of $\N_s^* (L_\ZZ(E), \D_\ZZ(E))$.  Note that
\begin{align*}
V^*V &= \left(\mu \mu^* - \sum_{ e \in S } (\mu e)(\mu e)^*\right) \mu \mu^* \left(\mu \mu^* - \sum_{ e \in S } (\mu e)(\mu e)^*\right) \otimes e_{1,1} \\
 &= \left(\mu \mu^* - \sum_{ e \in S } (\mu e)(\mu e)^*\right) \left(\mu \mu^* - \sum_{ e \in S } (\mu e)(\mu e)^*\right) \otimes e_{1,1} \\
 &= \left(\mu \mu^* - \sum_{ e \in S } (\mu e)(\mu e)^* \right) \otimes e_{1,1} \quad \text{and} \\
 VV^* &= \mu^*  \left(\mu \mu^* - \sum_{ e \in S }(\mu e)(\mu e)^*\right)\mu \otimes e_{1,1} = \left( r(\mu)  - \sum_{ e \in S } e e^* \right) \otimes e_{1,1},
\end{align*}
we have that 
\begin{align*}
\alpha_E( q_S ) &= \left[ \left( r(\mu) - \sum_{ e \in S} e e^* \right) \otimes e_{1,1} \right]_{ \D} = \left[ \left(\mu \mu^* - \sum_{ e \in S } (\mu e)(\mu e)^*\right) \otimes e_{1,1} \right]_{\D}.
\end{align*}
Thus, proving the claim.

We now show that $\alpha_E$ is a surjection.  Let $p$ be a projection in $\D$.  Then $p = \sum_{ k = 1}^n p_k \otimes e_{k,k}$, where $p_k$ is a projection in $\D_\ZZ(E)$.  Since $[p]_{ \D } = \sum_{ k = 1}^n [ p_k \otimes e_{1,1} ]_{ \D }$, to show that $\alpha_E$ is a surjective homomorphism, it is enough to prove that $q \otimes e_{1,1}$ is in the range of $\alpha_E$ for all projections $q$ in $\D_\ZZ(E)$.  Let $q$ be a projection in $\D_\ZZ(E)$.  By Lemma~\ref{lem:diag-proj-graph}, there are paths $\mu_1, \mu_2, \ldots, \mu_n$ in $E$ and finite sets $S_1, S_2, \ldots, S_n$ such that 
\begin{enumerate}
    \item each $S_i$ is a subset of $r(\mu_i)E^1$;

    \item $\left\{\mu_i \mu_i^* - \sum_{ e \in S_i } (\mu_i e)(\mu_i e)^* : i =1, 2, \ldots, n \right\}$ is a collection of mutually orthogonal projections; and 

    \item $q = \sum_{ i = 1}^n p_i$, where $p_i := \mu_i \mu_i^* - \sum_{ e \in S_i } (\mu_i e)(\mu_i e)^*$.
\end{enumerate}
Since $p_k$'s are mutually orthogonal projections, by Lemma~\ref{lem:11todiag}, 
$$[q \otimes e_{1,1} ]_{ \D } = \left[ \sum_{ k = 1}^n p_k \otimes  e_{k,k} \right]_{ \D } = \sum_{ k = 1}^n [ p_k \otimes  e_{1,1} ]_{ \D }.$$
By the above claim, $[ p_k \otimes  e_{1,1} ]_{ \D }$ is in the range of $\alpha_E$.  Thus, $[q \otimes e_{1,1} ]_{ \D }$ is in the range of $\alpha_E$.  We may now conclude that $\alpha_E$ is a surjection.

We have just shown the desired result that $\alpha_E$ is an isomorphism.  To prove the last statement, since $\beta_E \circ \alpha_E = \lambda_E$ and since $\lambda_E$ and $\alpha_E$ are isomorphisms, we have that $\beta_E$ is an isomorphism.
\end{proof}

We may also view $L_\ZZ(E)$ as a ring by forgetting its $*$-structure.  For this reason, we now define $\V(A,D)$ for a ring $A$ and a commutative subring $D$ of $A$.  Set 
\[
\N(A, D) := \{ (u,v) \in A \times A : u Dv \subseteq D, v Du \subseteq D \}.
\]
Observe that $\N^*(L_\ZZ(E), D_\ZZ(E)) \subseteq \N(L_\ZZ(E), D_\ZZ(E))$.

For idempotents $p$ and $q$ in $D$, we write $p \underset{D}{\sim} q$ provided that there are $u, v \in A$ such that $uv = p$, $vu = q$, and $(u,v) \in \N(A,D)$.  Then
\[
\V(A, D) := \{ [p]_{D^s} : p \text{ an idempotent in } D^s \}
\]
is an abelian monoid with addition defined by orthogonal sum.

\begin{cor}\label{cor:monoid}
For any graph $E$ and for any field $\mathsf{k}$, the inclusions 
\[
L_\ZZ(E) \hookrightarrow L_\mathsf{k}(E) \quad \text{and} \quad L_\ZZ(E) \hookrightarrow C^*(E)
\]
induce isomorphisms from $\V^*( L_\ZZ(E), \D_\ZZ(E))$ to $\V( L_\mathsf{k}(E), \D_\mathsf{k}(E))$ and from $\V^*( L_\ZZ(E), \D_\ZZ(E))$ to $\V^*( C^*(E), \D(E))$.  Moreover, the maps 
\[\theta_E \colon \V( L_\mathsf{k}(E), \D_\mathsf{k}(E)) \to \V(L_\mathsf{k}(E)), \ [e]_\D \mapsto [e] \] and 
\[\overline{\theta}_E \colon \V^*( C^*(E), \D(E)) \to  \V(C^*(E)), \ [p]_\D \mapsto [p]\] 
are isomorphisms.
\end{cor}

\begin{proof}
It is clear that $\iota_E \colon L_\ZZ(E) \to L_\mathsf{k}(E)$ induces a homomorphism $\alpha$ from $\V^*(L_\ZZ(E), \D_\ZZ(E))$ to $\V( L_\mathsf{k}(E), \D_\mathsf{k}(E))$ such that $\theta_E \circ \alpha = \beta_E$, where $\beta_E \colon \V^*(L_\ZZ(E), \D_\ZZ(E)) \to \V(L_\mathsf{k}(E))$ is the isomorphism defined in Theorem~\ref{thm:monoid}.  We immediately conclude that $\alpha$ is an injection.  Since idempotents of $\D_\mathsf{k}^s(E)$ are projections in $\D_\ZZ^s(E)$ (by Lemma~\ref{lem:diag-proj-graph}), we have $\alpha$ is a surjection.  Hence, $\alpha$ is an isomorphism which implies $\theta_E$ is an isomorphism.

The inclusion $\overline{\iota}_E \colon L_\ZZ(E) \to C^*(E)$ induces a homomorphism $\overline{\alpha}$ from $\V^*(L_\ZZ(E), \D_\ZZ(E))$ to $\V^*( C^*(E), \D(E))$ such that $\overline{\theta}_E \circ \overline{\alpha} = j_E \circ \beta_E$, where $j_E$ is the homomorphism from $\V(L_\CC(E))$ to $\V(C^*(E))$ induced by the inclusion of $L_\CC(E)$ into $C^*(E)$.  By \cite[Corollary~3.5]{HLMRT}, $j_E$ is an isomorphism.  Hence, $\overline{\alpha}$ is an injection.  By \cite[Proposition~4.1]{KPRR}, there exists a $*$-isomorphism from $C^*(E)$ to $C^*(\mathcal{G}_E)$ sending $\D(E)$ to $C_0( \mathcal{G}_E^0)$ such that $\mu \mu^*$ is sent to $1_{\Z(\mu)}$.  Hence, arguing as in Lemma~\ref{lem:diag-proj-graph}, each projection in $\D(E)$ is the image of a projection in $L_\ZZ(E)$.  Thus, each projection in $\D^s(E)$ is the image of a projection in $\D_\ZZ^s(E)$.  So, $\overline{\alpha}$ is a surjection.  Consequently, $\overline{\alpha}$ is an isomorphism which implies $\overline{\theta}_E$ is an isomorphism.
\end{proof}

Recall that for a graph $E$, $\Sigma E$ is the graph where for each $v \in E^0$, we attach an infinite fan at $v$,
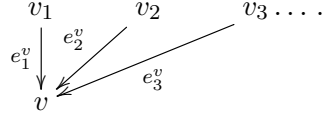
\[
\xymatrix{ v_1 \ar[d]_-{e_1^v} &  \ar[dl]_-{e_2^v} v_2  & \ar[dll]^-{e_3^v} v_3\ldots    \\ 
         v}.
\]
Then there exists a $*$-isomorphism $\psi$ from $L_\ZZ(\Sigma E)$ to $\mathsf{M}_\NN(L_\ZZ(E))$ given by 
\[ w \mapsto \begin{cases} w &\text{ if } w \in E^0 \\ v \otimes e_{\sigma(n), \sigma(n)} &\text{ if } w = v_n  \end{cases} \quad \text{and} \quad e \mapsto \begin{cases} e \otimes e_{1,1} &\text{ if } e \in E^1 \\ v \otimes e_{\sigma(n), 1 } &\text{ if } e = e_n^v \end{cases} \] where $\sigma \colon \NN \to \NN\setminus \{1\}$ is any bijection.  By \cite{Stein2024}, $\psi$ is diagonal-preserving, which then induces a diagonal-preserving isomorphism $\psi_E$ from $L_\mathsf{k}(\Sigma E)$ to $\mathsf{M}_\NN(L_\mathsf{k}(E))$ and a diagonal-preserving isomorphism $\overline{\psi}_E$ from $C^*(\Sigma E)$ to $C^*(E) \otimes \Kk$.  Moreover, the induced maps on the $K_0$-groups are the identity maps on $E^0$.

\begin{proof}[Proof of Theorem~\ref{thm-intro}]
The fact that \eqref{EF-cstar}, \eqref{EF-LPA}, \eqref{EF-LPA-Z}, \eqref{EF-Groupoid}, and \eqref{EF-fullgrps} are equivalent statements follows from \cite[Corollary~4.3 and Corollary~4.4]{ABHS}, \cite[Theorem~1]{TCadv}, and \cite[Theorem~D]{NO2019}.  

To prove that \eqref{EF-cstar} is equivalent to \eqref{cartan}, we only need to prove \eqref{EF-cstar} implies \eqref{cartan}. 
Assume \eqref{EF-cstar} is true and assume that $\mathcal{O}_\infty \cong C^*(H)$.  By \cite[Proposition~5.2.1]{ER-Refined}, we may assume $H$ is obtained from an essential graph $H_0$ with a finite number of sources added to it and $|H_0| \geq 2$.  Outsplitting $E(\infty)$ using a partition of size $|H_0^0|$, we get a graph $E_\infty$ for which there is a diagonal-preserving isomorphism between $\mathcal{O}_\infty$ and $C^*(E_\infty)$.  Outspliting $E(\infty, -)$ using a partition of size $|H_0^0|-2$, we get an essential graph $G_{\infty,-}$ and a graph $E_{\infty, -}$ for which $E_{\infty, -}$ is obtained from $G_{\infty,-}$ by adding a finite number of sources to $G_{\infty,-}$ and there is a diagonal-preserving isomorphism between $\mathcal{O}_{\infty, -}$ and $C^*(E_{\infty, -})$.   

Set 
\begin{align*}
\mathbf{v}_\infty &= (\overbrace{1, \ldots, 1}^{|H_0^0|})^T, &\mathbf{w} &=(1 +n_1, \ldots, 1 + n_{|H_0^0|})^T, \\
\mathbf{v}_{\infty, -} &=(1 +m_1, \ldots, 1 + m_{|G_{\infty,-}^0|})^T \\
\end{align*}
where $m_{i}$ is the number of sources to the vertex $i$ in $G_{\infty,-}$ and $n_i$ is the number of sources to the vertex $i$ in $H_0$.  By \cite[Lemma~7.1.6]{ER-Refined}, there are invertible matrices $U, V$ with $\det(U)=-1$ and $\det(V)=1$ such that 
\[
U( \mathsf{A}_{E_\infty}^\bullet - \mathsf{I})^T V = (\mathsf{A}_{G_{\infty,-}}^\bullet - \mathsf{I})^T
\]
and $U \mathbf{v}_{\infty}= \mathbf{v}_{\infty,-}$ in $\operatorname{coker}(\mathsf{A}_{G_{\infty,-}}^\bullet - \mathsf{I})^T $.  Here, $\mathsf{A}_E^\bullet$ is the rectangular matrix obtained from the adjacency matrix of a graph $E$ with all rows corresponding to singular vertices removed.  Since $\mathcal{O}_\infty \cong C^*(H)$, by \cite[Theorem~8.6]{ERRSDuke} and \cite[Lemma~7.1.6]{ER-Refined}, there are invertible matrices $W, Z$ with $\det(W)= \pm 1$ and $\det(Z) = 1$ such that 
\[
W(\mathsf{A}_{E_\infty}^\bullet - \mathsf{I})^T Z = (\mathsf{A}_{H_0}^\bullet - \mathsf{I})^T
\]
and $W \mathbf{v}_\infty = \mathbf{w}$ in $\operatorname{coker}(\mathsf{A}_{H_0}^\bullet - \mathsf{I})^T$.  

If $\det(W)=1$, then by \cite[Theorem~5.4]{ARS}, there exists a diagonal-preserving isomorphism between $C^*(E_\infty)$ and $C^*(H)$.  Consequently, there exists a diagonal-preserving isomorphism between $\mathcal{O}_\infty$ and $C^*(H)$.  Assume that $\det(W)=-1$.  Then 
\[
UW^{-1}(\mathsf{A}_{H_0}^\bullet - \mathsf{I})^TZ^{-1}V = (\mathsf{A}_{G_{\infty,-}}^\bullet - \mathsf{I})^T
\]
such that $UW^{-1} \mathbf{w} = \mathbf{v}_{\infty,-}$ in $\operatorname{coker}(\mathsf{A}_{G_{\infty,-}}^\bullet - \mathsf{I})^T$.  Since $\det(UW^{-1} ) = \det(Z^{-1}V)=1$, by \cite[Theorem~5.4]{ARS}, there exists a diagonal-preserving isomorphism between $C^*(H)$ and $C^*(E_{\infty,-})$.  Hence, we get a diagonal-preserving isomorphism between $C^*(H)$ and $C^*(E(\infty,-))$.  Thus, concluding the proof that \eqref{EF-cstar} is equivalent to \eqref{cartan}.

To prove the rest of the statements are equivalent we first argue that we may replace $E(\infty)$ and $E(\infty, -)$ with the following graphs  \[
E:\qquad 
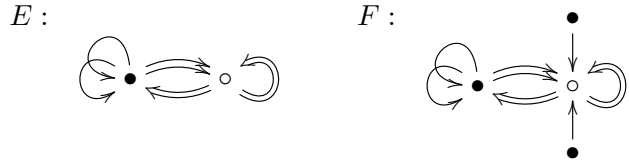
\xymatrix@R=5mm{&\\\bullet\ar@(ul,dl)[]\ar@(u,l)[]\ar@/^/@<-0.5mm>[r]\ar@/^/@<0.5mm>[r]&\circ\ar@{=>}@/^/[l]\ar@{=>}@(dr,ur)[]}\qquad\qquad F:\qquad
\xymatrix@R=5mm{&\bullet\ar[d]\\\bullet\ar@(ul,dl)[]\ar@(u,l)[]\ar@/^/@<-0.5mm>[r]\ar@/^/@<0.5mm>[r]&\circ\ar@{=>}@/^/[l]\ar@{=>}@(dr,ur)[]\\&\bullet\ar[u]}
\]
Outsplitting the regular vertex of $F$ with two loops using two sets in the partition with the loops in different partitions, by \cite[Lemma~7.1.6]{ER-Refined} and \cite[Theorem~5.4]{ARS}, there exists a diagonal-preserving isomorphism between the graph $C^*$-algebra of the resulting graph and $C^*(E(\infty,-))$.  Since $E$ is obtained from $E(\infty)$ by outsplitting the graph $E(\infty)$, by \cite[Theorems~2.1.2 and 4.2.2]{ER-Refined}, there exists a diagonal-preserving isomorphism between $C^*(E)$ and $C^*(E(\infty))$.  Hence, we may replace $E(\infty)$ with $E$ and we may replace $E(\infty, -)$ with $F$ in Theorem~\ref{thm-intro}.

We now prove that \eqref{EF-cstar} and \eqref{Oinfty} are equivalent.  Suppose that there exists a diagonal-preserving isomorphism from $C^*(E)$ to $C^*(F)$.  So, we have a diagonal-preserving isomorphism $\varphi_1$ from $C^*(E) \otimes \Kk$ to $C^*(F) \otimes \Kk$ that sends $[1_{C^*(E)} \otimes e_{1,1}]$ to $[p_v\otimes e_{1,1}] + 3 [p_w\otimes e_{1,1}]$, where $v$ is the regular vertex of $E^0 \subseteq F^0$ and $w$ is the singular vertex of $E^0 \subseteq F^0$.  Using the graph move ``unital reduction" \cite[Theorem~2.3.2]{ER-Refined}, we have a diagonal-preserving isomorphism $\varphi_2$ from $C^*(\Sigma F)$ to $C^*(\Sigma E)$ such that $\varphi_2( p_v)=p_v$ and $\varphi_2(p_w)=p_w$ for $v , w \in E^0 \subseteq F^0$.  Then $\overline{\psi}_E \circ \varphi_2 \circ (\overline{\psi}_F)^{-1} \circ \varphi_1$ is a diagonal-preserving automorphism of $C^*(E) \otimes \Kk$ sending $[1_{C^*(E)} \otimes e_{1,1}]$ to $[p_v\otimes e_{1,1}] + 3[p_w\otimes e_{1,1}]$.  Since 
\begin{align*}[1_{C^*(E)} \otimes e_{1,1} ] &= [p_v \otimes e_{1,1}] + [p_v \otimes e_{1,1}] + 3[p_w \otimes e_{1,1}] \\ &= 3[p_v \otimes e_{1,1}] + 5[p_w \otimes e_{1,1}],\end{align*}
we have $[p_v \otimes e_{1,1}]+3[p_w \otimes e_{1,1}] = -[1_{C^*(E)}\otimes e_{1,1}]$.  In particular, the induced map in $K$-theory is the non-trivial automorphism of $K_0(C^*(E))$.  

Suppose there exists a diagonal-preserving automorphism $\lambda$ on $C^*(E)\otimes \Kk$ which induces the non-trivial automorphism on $K_0(C^*(E))$.  Then $\alpha=\overline{\psi}_F \circ \varphi_2^{-1} \circ (\overline{\psi}_E)^{-1}\circ \lambda$ is a diagonal-preserving isomorphism from $C^*(E) \otimes \Kk$ to $C^*(F) \otimes \Kk$.  Since $-[1_{C^*(E)}\otimes e_{1,1}]=[p_v \otimes e_{1,1}]+3[p_w \otimes e_{1,1}]$, $K_0(\alpha)$ sends $[1_{C^*(E)} \otimes e_{1,1} ]$ to $[p] = [p_v \otimes e_{1,1}]+3[p_w \otimes e_{1,1}] = [ 1_{C^*(F)} \otimes e_{1,1} ]$, where $p \in \D(E) \otimes c_0(\NN)$.  Since $C^*(F)$ is a simple purely infinite $C^*$-algebra, $[p] = [ 1_{C^*(F)} \otimes e_{1,1} ]$ in $\V(C^*(F))$.  By Corollary~\ref{cor:monoid} and by the diagonal-preserving isomorphism between $C^*(\Sigma F)$ and $C^*(F) \otimes \Kk$, there exists $v \in \N^*( C^*(F)\otimes \Kk, \mathcal{D}(F) \otimes c_0(\NN))$ such that $v^*v = p$ and $vv^*=1_{C^*(F)} \otimes e_{1,1}$.  Then the map $a \in C^*(E) \mapsto v \alpha(a \otimes e_{1,1}) v^*$ is a diagonal-preserving isomorphism from $C^*(E)$ to $C^*(F)$.

We have just completed the proof that \eqref{EF-cstar} and \eqref{Oinfty} are equivalent.  The proof of the equivalence of \eqref{EF-LPA} and \eqref{Linfty} uses a similar argument, where the isomorphism between $\V( L_\mathsf{k}(E), \D_\mathsf{k}(E))$ and $\V(L_\mathsf{k}(E))$ is used instead.

For 
\[\eqref{Oinfty} \implies \eqref{Groupoid} \implies \eqref{Linfty-Z} \implies \eqref{Linfty},\]we first make the following observations.  By \cite[Theorem~6.5 and the proof of Theorem~6.6]{HazLi2022}, there are isomorphisms from $\D_\ZZ(E)/\langle r(\mu)-\mu \mu^* : \mu \in E^*\rangle$ to $K_0(L_\mathsf{k}(E))$, from $\D_\ZZ(E)/\langle r(\mu)-\mu \mu^* : \mu \in E^*\rangle$ to $H_0(\mathcal{G}_E) := C_c( \partial E, \ZZ)/ \operatorname{Im} \partial_1$, and from $H_0(\mathcal{G}_E)$ to $K_0(L_\mathsf{k}(E))$ such that the diagram 
\begin{equation}\label{kthy-homology}
\vcenter{\xymatrix{ \D_\ZZ(E)/\langle r(\mu)-\mu \mu^* : \mu \in E^*\rangle \ar[r] \ar[d]  & K_0(L_\mathsf{k}(E)) \\
H_0(\mathcal{G}_E) \ar[ru] }}
\end{equation}
is commutative, where the vertical isomorphism is induced by the isomorphism between $\D_\ZZ(E)$ and $C_c( \partial E, \ZZ)$ that sends $\mu \mu^*$ to $1_{ \mathcal{Z}(\mu)}$.  

Suppose there exists a diagonal-preserving automorphism on $\mathsf{M}_\NN(L_\mathsf{k}(E))$ that induces the non-trivial automorphism on $K_0(L_\mathsf{k}(E))$.  Then there exists a diagonal-preserving automorphism $\varphi$ on $L_\mathsf{k}(\Sigma E)$ that induces the non-trivial automorphism on $K_0(L_\mathsf{k}(\Sigma E))$.  By \cite[Theorem~6.1]{Stein2019} and its proof, there exists a groupoid automorphism $\alpha$ on $G_{\Sigma E}$ such that the induced automorphism on $L_\mathsf{k}(\Sigma E)$ restricted to the diagonal is $\varphi$.  By Diagram~\eqref{kthy-homology}, $\alpha$ induces the non-trivial automorphism on zeroth groupoid homology of $\mathcal{G}_{\Sigma E}$.  Thus, \eqref{Oinfty} implies \eqref{Groupoid}.  

Suppose $\alpha$ is a groupoid automorphism of $\mathcal{G}_{\Sigma E}$ that induces the non-trivial automorphism on zeroth groupoid homology of $\mathcal{G}_{\Sigma E}$.  Then $\alpha$ induces a diagonal-preserving $*$-automorphism $L_\ZZ(\Sigma E)$.  Applying the K-theory functor to this $*$-automorphism, by Diagram~\eqref{kthy-homology}, we get the non-trivial automorphism on $K_0(L_\ZZ(\Sigma E))$.  Since there exists a diagonal-preserving $*$-isomorphism between $L_\ZZ(\Sigma E)$ and $\mathsf{M}_\NN(L_\ZZ(E))$ for which $[z]$ is sent to $[z \otimes e_{1,1}]$ in $K_0$ for all vertices $z$ of $E$, we have a $*$-automorphism $\mathsf{M}_\NN(L_\ZZ(E))$ that induces the non-trivial automorphism on $K_0( L_\ZZ(E))$.  Hence, \eqref{Groupoid} implies \eqref{Linfty-Z}.

Lastly, we prove \eqref{Linfty-Z} implies \eqref{Linfty}.  Suppose $\alpha$ is a $*$-automorphism on $\mathsf{M}_\NN(L_\ZZ(E))$ that induces the non-trivial automorphism on $K_0( L_\ZZ(E))$.  By tensoring $\alpha$ with $\mathrm{id}_\mathsf{k}$, we get a diagonal-preserving automorphism on $\mathsf{M}_\NN(L_\mathsf{k}(E))$ inducing the non-trivial automorphism on $K_0(L_\mathsf{k}(E))$.  Thus, completing the proof of the desired equivalences.
\end{proof}

\begin{rmk}
In \cite[Question~7.1.8]{ER-Refined}, we asked whether there is a diagonal-preserving isomorphism between $C^*(E)$ and $C^*(F)$, where 
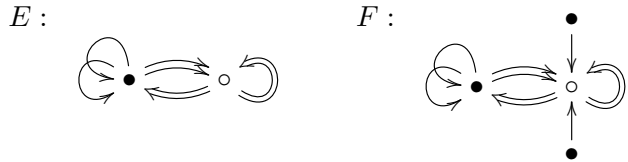
\[
E:\qquad \xymatrix@R=5mm{&\\\bullet\ar@(ul,dl)[]\ar@(u,l)[]\ar@/^/@<-0.5mm>[r]\ar@/^/@<0.5mm>[r]&\circ\ar@{=>}@/^/[l]\ar@{=>}@(dr,ur)[]}\qquad\qquad F:\qquad
\xymatrix@R=5mm{&\bullet\ar[d]\\\bullet\ar@(ul,dl)[]\ar@(u,l)[]\ar@/^/@<-0.5mm>[r]\ar@/^/@<0.5mm>[r]&\circ\ar@{=>}@/^/[l]\ar@{=>}@(dr,ur)[]\\&\bullet\ar[u]}
\]
These graphs came about during the process of determining whether there exists a diagonal-preserving isomorphism between $\mathcal{O}_\infty$ and $\mathcal{O}_{\infty,-}$.  In the proof of Theorem~\ref{thm-intro} we showed the question of whether there exists a diagonal-preserving isomorphism between $\mathcal{O}_\infty$ and $\mathcal{O}_{\infty,-}$ is equivalent to the question of whether there exists a diagonal-preserving isomorphism between $C^*(E)$ and $C^*(F)$.  The simplicity of the construction of $F$ from $E$ suggest that the problem of determining when there is a diagonal-preserving isomorphism between graph algebras is a delicate one.
\end{rmk}

\end{document}